\documentclass{amsart}
\usepackage[all]{xy}
\usepackage{verbatim}
\usepackage{color}
\usepackage{amsthm}
\usepackage{amssymb}
\usepackage[colorlinks=true]{hyperref}

\numberwithin{equation}{section}

\newtheorem{theorem}[equation]{Theorem}
\newtheorem*{theorem*}{Theorem}

\newtheorem*{conjecture*}{Mamma Conjecture}
\newtheorem*{conjecture1*}{Mamma Conjecture (revisited)}
\newtheorem{proposition}[equation]{Proposition}

\newtheorem*{corollary*}{Corollary}

\theoremstyle{remark}

\theoremstyle{remark}
\newtheorem{remark}[equation]{Remark}

\newcommand{\cB}{{\mathcal B}}
\newcommand{\cC}{{\mathcal C}}

\newcommand{\cL}{{\mathcal L}}

\newcommand{\cO}{{\mathcal O}}

\newcommand{\cX}{{\mathcal X}}

\newcommand{\bbA}{\mathbb{A}}

\newcommand{\bbC}{\mathbb{C}}

\newcommand{\bbG}{\mathbb{G}}

\newcommand{\bbP}{\mathbb{P}}

\newcommand{\bbQ}{\mathbb{Q}}
\newcommand{\bbZ}{\mathbb{Z}}

\DeclareMathOperator{\id}{id}

\newcommand{\dgcat}{\mathrm{dgcat}} 

\newcommand{\perf}{\mathrm{perf}}

\newcommand{\dg}{\mathrm{dg}}

\newcommand{\uHom}{\underline{\mathrm{Hom}}}
\newcommand{\Hom}{\mathrm{Hom}}

\newcommand{\too}{\longrightarrow}

\newcommand{\ie}{\textsl{i.e.}\ }

\let\oldmarginpar\marginpar
\def\marginpar#1{\oldmarginpar{\tiny #1}}

\begin{document}

\title[Kimura-finiteness of quadric fibrations over smooth curves]{Kimura-finiteness of \\quadric fibrations over smooth curves}
\author{Gon{\c c}alo~Tabuada}

\address{Gon{\c c}alo Tabuada, Department of Mathematics, MIT, Cambridge, MA 02139, USA}
\email{tabuada@math.mit.edu}
\urladdr{http://math.mit.edu/~tabuada}
\thanks{The author was partially supported by a NSF CAREER Award}

\date{\today}


\abstract{In this short note, making use of the recent theory of noncommutative mixed motives, we prove that the Voevodsky's mixed motive of a quadric fibration over a smooth curve is Kimura-finite.}
}

\maketitle
\vskip-\baselineskip
\vskip-\baselineskip



\section{Introduction}
 Let $(\cC,\otimes, {\bf 1})$ be a $\bbQ$-linear, idempotent complete, symmetric monoidal category. Given a partition $\lambda$ of an integer $n\geq 1$, consider the corresponding irreducible $\bbQ$-linear representation $V_\lambda$ of the symmetric group $\mathfrak{S}_n$ and the associated idempotent $e_\lambda\in \bbQ[\mathfrak{S}_n]$. Under these notations, the Schur-functor $S_\lambda\colon \cC \to \cC$ sends an object $a$ to the direct summand of $a^{\otimes n}$ determined by $e_\lambda$. In the particular case of the partition $\lambda=(1,\ldots, 1)$, resp. $\lambda=(n)$, the associated Schur-functor $\wedge^n:=S_{(1,\ldots, 1)}$, resp. $\mathrm{Sym}^n:=S_{(n)}$, is called the {\em $n^{\mathrm{th}}$ wedge product}, resp. the {\em $n^{\mathrm{th}}$ symmetric~product}. Following Kimura \cite{Kimura}, an object $a \in \cC$ is called {\em even-dimensional}, resp. {\em odd-dimensional}, if $\wedge^n(a)$, resp. $\mathrm{Sym}^n(a)=0$, for some $n \gg 0$. The biggest integer $\mathrm{kim}_+(a)$, resp. $\mathrm{kim}_-(a)$, for which $\wedge^{\mathrm{kim}_+(a)}\neq 0$, resp. $\mathrm{Sym}^{\mathrm{kim}_-(a)}(a)\neq 0$, is called the {\em even}, resp. {\em odd}, {\em Kimura-dimension of $a$}. An object $a \in \cC$ is called {\em Kimura-finite} if $a\simeq a_+\oplus a_-$, with $a_+$ even-dimensional and $a_-$ odd-dimensional. The integer $\mathrm{kim}(a)=\mathrm{kim}_+(a_+)+\mathrm{kim}_-(a_-)$ is called the {\em Kimura-dimension of $a$}.
 
Voevodsky introduced in \cite{Voevodsky} an important triangulated category of geometric mixed motives $\mathrm{DM}_{\mathrm{gm}}(k)_\bbQ$ (over a perfect base field $k$). By construction, this category is $\bbQ$-linear, idempotent complete, rigid symmetric monoidal, and comes equipped with a symmetric monoidal functor $M(-)_\bbQ \colon \mathrm{Sm}(k) \to \mathrm{DM}_{\mathrm{gm}}(k)_\bbQ$, defined on smooth $k$-schemes. An important open problem\footnote{Among other consequences, Kimura-finiteness implies rationality of the motivic zeta function.} is the classification of all the Kimura-finite mixed motives and the computation of the corresponding Kimura-dimensions. On the negative side, O'Sullivan constructed a certain smooth surface $S$ whose mixed motive $M(S)_\bbQ$ is {\em not} Kimura-finite; consult \cite[\S5.1]{Mazza} for details. On the positive side, Guletskii \cite{Guletskii} and Mazza \cite{Mazza} proved, independently, that the mixed motive $M(C)_\bbQ$ of every smooth curve $C$ is Kimura-finite. 

The following result bootstraps Kimura-finiteness from smooth curves to families of quadrics over smooth curves:
\begin{theorem}\label{thm:Kimura}
Let $k$ be a field, $C$ a smooth $k$-curve, and $q\colon Q \to C$ a flat quadric fibration of relative dimension $d-2$. Assume that $Q$ is smooth and that $q$ has only {\em simple degenerations}, \ie  that all the fibers of $q$ have corank $\leq 1$. 
\begin{itemize}
\item[(i)] When $d$ is even, the mixed motive $M(Q)_\bbQ$ is Kimura-finite. Moreover, we have
$$\mathrm{kim}(M(Q)_\bbQ) = \mathrm{kim}(M(\widetilde{C})_\bbQ) + (d-2) \mathrm{kim}(M(C)_\bbQ)\,,$$
where $D \hookrightarrow C$ stands for the finite set of critical values of $q$ and $\widetilde{C}$ for the discriminant double cover of $C$ (ramified over $D$).
\item[(ii)] When $d$ is odd, $k$ is algebraically closed, and $1/2 \in k$, the mixed motive $M(Q)_\bbQ$ is Kimura-finite. Moreover, we have the following equality:
$$\mathrm{kim}(M(Q)_\bbQ) = \# D + (d-1) \mathrm{kim}(M(C)_\bbQ)\,.$$
\end{itemize}
\end{theorem}
To the best of the authors' knowledge, Theorem \ref{thm:Kimura} is new in the literature. It not only provides new (families of) examples of Kimura-finite mixed motives but also computes the corresponding Kimura dimensions.
\begin{remark}
In the particular case where $k$ is algebraically closed and $Q, C$ are moreover projective, Vial proved in \cite[Cor.~4.4]{Vial} that the Chow motive $\mathfrak{h}(Q)_\bbQ$ is Kimura-finite. Since the category of Chow motives embeds fully-faithfully into $\mathrm{DM}_{\mathrm{gm}}(k)_\bbQ$ (see \cite[\S4]{Voevodsky}), we then obtain in this particular case an alternative ``geometric'' proof of the Kimura-finiteness of $M(Q)_\bbQ$. Moreover, when $k=\bbC$ and $d$ is odd, Bouali refined Vial's work by showing that $\mathfrak{h}(Q)_\bbQ$ is isomorphic to $\bbQ(-\frac{d-1}{2})^{\oplus \# D} \oplus \bigoplus_{i=0}^{d-2} \mathfrak{h}(C)_\bbQ(-i)$; see \cite[Rk.~1.10(i)]{Bouali}. In this particular case, this leads to an alternative ``geometric'' computation of the Kimura-dimension~of~$M(Q)_\bbQ$.
\end{remark}
\section{Preliminaries}
In what follows, $k$ denotes a base field.
\subsection*{Dg categories}
For a survey on dg categories consult Keller's ICM talk \cite{ICM-Keller}. 
In what follows, we write $\dgcat(k)$ for the category of (small) dg categories and dg functors. Every (dg) $k$-algebra gives naturally rise to a dg category with a single object. Another source of examples is provided by schemes/stacks since the category of perfect complexes $\perf(X)$ of every $k$-scheme $X$ (or, more generally, algebraic stack $\cX$) admits a canonical dg enhancement $\perf_\dg(X)$; see \cite[\S4.6]{ICM-Keller}\cite{LO}.
\subsection*{Noncommutative mixed motives}
For a book, resp. survey, on noncommutative motives consult \cite{book}, resp. \cite{survey}. Recall from~\cite[\S8.5.1]{book} the construction of Kontsevich's triangulated category of noncommutative mixed motives $\mathrm{NMot}(k)$; denoted by $\mathrm{NMot}^{\bbA^1}_{\mathrm{loc}}(k)$ in {\em loc. cit.} By construction, this category is idempotent complete, closed symmetric monoidal, and comes equipped with a symmetric monoidal functor $U \colon \dgcat(k) \to \mathrm{NMot}(k)$.
\subsection*{Root stacks}
Let $X$ be a $k$-scheme, $\cL$ a line bundle on $X$, $\sigma \in \Gamma(X,\cL)$ a global section, and $r>0$ an integer. In what follows, we write $D\hookrightarrow X$ for the zero locus of $\sigma$. Recall from \cite[Def.~2.2.1]{Codman} (see also \cite[Appendix B]{GW}) that the associated {\em root stack} is defined as the following fiber-product of algebraic stacks
$$
\xymatrix{
\sqrt[r]{(\cL,\sigma)/X} \ar[d] \ar[r] & [\bbA^1/\bbG_m] \ar[d]^-{\theta_r} \\
X \ar[r]_-{(\cL,\sigma)} & [\bbA^1/\bbG_m]\,,
}
$$
where $\theta_r$ stands for the morphism induced by the $r^{\mathrm{th}}$ power maps on $\bbA^1$ and $\bbG_m$.
\begin{proposition}\label{prop:aux}
We have an isomorphism $U(\sqrt[r]{(\cL,\sigma)/X})\simeq U(D)^{\oplus (r-1)}\oplus U(X)$ whenever $X$ and $D$ are $k$-smooth.
\end{proposition}
\begin{proof}
By construction, the root stack comes equipped with a forgetful morphism $f\colon \sqrt[r]{(\cL,\sigma)/X} \to X$. As proved by Ishii-Ueda in \cite[Thm.~1.6]{Ueda}, the pull-back functor $f^\ast$ is fully-faithful. Moreover, we have a semi-orthogonal decomposition 
$$ \perf(\cX) = \langle \perf(D)_{r-1}, \ldots, \perf(D)_1, f^\ast(\perf(X)) \rangle\,,$$
where all the categories $\perf(D)_i$ are equivalent (via a Fourier-Mukai type functor) to $\perf(D)$. Consequently, the proof follows from the fact that the functor $U$ sends semi-orthogonal decomposition to direct sums (see \cite[\S8.4.1 and \S8.4.5]{book}).
\end{proof}
\subsection*{Orbit categories}
Let $(\cC, \otimes, {\bf 1})$ be an $\bbQ$-linear symmetric monoidal additive category and $\cO \in \cC$ a $\otimes$-invertible object. The {\em orbit category} $\cC/_{\!-\otimes \cO}$ has the same objects as $\cC$ and morphisms $\Hom_{\cC/_{\!-\otimes \cO}}(a,b):=\oplus_{n \in \bbZ} \Hom_\cC(a, b \otimes \cO^{\otimes n})$. Given objects $a, b, c$ and morphisms $\mathrm{f}=\{f_n\}_{n \in \bbZ}$ and $\mathrm{g}=\{g_n\}_{n \in \bbZ}$, the $i^{\mathrm{th}}$-component of $\mathrm{g}\circ \mathrm{f}$ is defined as $\sum_n (g_{i -n} \otimes \cO^{\otimes n})\circ f_n$. The canonical functor $\pi\colon \cC \to \cC/_{\!-\otimes \cO}$, given by $a \mapsto a$ and $f \mapsto \mathrm{f}=\{f_n\}_{n \in \bbZ}$, where $f_0=f$ and $f_n=0$ if $n\neq 0$, is endowed with an isomorphism $\pi \circ (-\otimes \cO) \Rightarrow \pi$ and is $2$-universal among all such functors. Finally, the category $\cC/_{\!-\otimes \cO}$ is $\bbQ$-linear, additive, and inherits from $\cC$ a symmetric monoidal structure making $\pi$ symmetric monoidal.
\section{Proof of Theorem \ref{thm:Kimura}}
Following Kuznetsov \cite[\S3]{Quadrics} (see also Auel-Bernardara-Bolognesi \cite[\S1.2]{ABB}), let $E$ be a vector bundle of rank $d$ on $C$, $p\colon \bbP(E) \to C$ the projectivization of $E$ on $C$, $\cO_{\bbP(E)}(1)$ the Grothendieck line bundle on $\bbP(E)$, $\cL$ a line bundle on $C$, and finally $\rho\in \Gamma(C, S^2(E^\vee) \otimes \cL^\vee) = \Gamma(\bbP(E), \cO_{\bbP(E)}(2)\otimes \cL^\vee)$ a global section. Given this data, $Q\subset \bbP(E)$ is defined as the zero locus of $\rho$ on $\bbP(E)$ and $q\colon Q \to C$ as the restriction of $p$ to $Q$; the relative dimension of $q$ is equal to $d-2$. Consider also the discriminant global section $\mathrm{disc}(q) \in \Gamma(C, \mathrm{det}(E^\vee)^{\otimes 2} \otimes (\cL^\vee)^{\otimes d})$ and the associated zero locus $D\hookrightarrow C$. Note that $D$ agrees with the finite set of critical values of $q$. Recall from \cite[\S3.5]{Quadrics}(see also \cite[\S1.6]{ABB}) that, when $d$ is even we have a discriminant double cover $\widetilde{C}$ of $C$ ramified over $D$. Moreover, since by hypothesis $q$ has only simple degenerations, $\widetilde{C}$ is $k$-smooth. Under the above notations, we have the following computation:
\begin{proposition}\label{prop:computation}
Let $q\colon Q \to C$ be a flat quadric fibration as above.
\begin{itemize}
\item[(i)] When $d$ is even, we have an isomorphism $U(Q)_{\bbZ[1/2]} \simeq U(\widetilde{C})_{\bbZ[1/2]} \oplus U(C)_{\bbZ[1/2]}^{\oplus (d-2)}$.
\item[(ii)] When $d$ is odd, $k$ is algebraically closed, and $1/2\in k$, we have an isomorphism $U(Q) \simeq U(D) \oplus U(C)^{\oplus (d-1)}$.
\end{itemize}
\end{proposition}
\begin{proof}
Recall from \cite[\S3]{Quadrics} (see also \cite[\S1.5]{ABB}) the construction of the sheaf $\cC_0$ of even parts of the Clifford algebra associated to $q$. As proved in \cite[Thm.~4.2]{Quadrics} (see also \cite[Thm.~2.2.1]{ABB}), we have a semi-orthogonal decomposition
$$ \perf(Q) = \langle \perf(C; \cC_0), \perf(C)_1, \ldots, \perf(C)_{d-2}\rangle\,,$$
where $\perf(C; \cC_0)$ stands for the category of perfect $\cC_0$-modules and $\perf(C)_i:=q^\ast(\perf(C)) \otimes \cO_{Q/C}(i)$. Note that all the categories $\perf(C)_i$ are equivalent (via a Fourier-Mukai type functor) to $\perf(C)$. Since the functor $U$ sends semi-orthogonal decompositions to direct sums, we then obtain a direct sum decomposition
\begin{equation}\label{eq:decomp}
U(Q) \simeq U(\perf^\dg(C;\cC_0))\oplus U(C)^{\oplus (d-2)}\,,
\end{equation}
where $\perf^\dg(C;\cC_0)$ stands for the dg enhancement of $\perf(C;\cC_0)$ induced from $\perf_\dg(Q)$. As explained in \cite[Prop.~4.9]{Quadrics} (see also \cite[\S2.2]{ABB}), the inclusion of categories $\perf(C;\cC_0) \hookrightarrow \perf(Q)$ is of Fourier-Mukai type. Therefore, the associated kernel leads to a Fourier-Mukai Morita equivalence between $\perf^\dg(C;\cC_0)$ and $\perf_\dg(C;\cC_0)$. Consequently, we can replace the dg category $\perf^\dg(C,\cC_0)$ by $\perf_\dg(C;\cC_0)$ in the above decomposition \eqref{eq:decomp}.

{\bf Item (i).} As explained in \cite[\S3.5]{Quadrics} (see also \cite[\S1.6]{ABB}), the category $\perf(C;\cC_0)$ is equivalent (via a Fourier-Mukai type functor) to $\perf(\widetilde{C}; \cB_0)$, where $\cB_0$ is a certain sheaf of Azumaya algebras over $\widetilde{C}$ of rank $2^{(d/2)-1}$. Therefore, the associated kernel leads to a Fourier-Mukai equivalence between $\perf_\dg(C;\cC_0)$ and $\perf_\dg(\widetilde{C}; \cB_0)$. As proved in \cite[Thm.~2.1]{Azumaya}, since $\cB_0$ is a sheaf of Azumaya algebras of rank $2^{(d/2)-1}$, the noncommutative mixed motive $U(\perf_\dg(\widetilde{C}; \cB_0))_{\bbZ[1/2]}$ is canonically isomorphic to $U(\widetilde{C})_{\bbZ[1/2]}$. Consequently, the $\bbZ[1/2]$-linearization of the right-hand side of \eqref{eq:decomp} reduces to $U(\widetilde{C})_{\bbZ[1/2]}\oplus U(C)_{\bbZ[1/2]}^{\oplus (d-2)}$. 

{\bf Item (ii).} As explained in \cite[Cor.~3.16]{Quadrics} (see also \cite[\S1.7]{ABB}), since by assumption $k$ is algebraically closed and $1/2 \in k$, the category $\perf(C;\cC_0)$ is equivalent (via a Fourier-Mukai type functor) to $\perf(\cX)$. This implies that the dg category $\perf_\dg(C;\cC_0)$ is Morita equivalent to $\perf_\dg(\cX)$. Consequently, since $C$ and $D$ are $k$-smooth, we conclude from the above Proposition \ref{prop:aux} that the right-hand side of \eqref{eq:decomp} reduces to $U(D)\oplus U(C)^{\oplus (d-1)}$. 
\end{proof}
{\bf Item (i).}  As proved in \cite[Thm.~2.8]{Bridge}, there exists a $\bbQ$-linear, fully-faithful, symmetric monoidal functor $\Phi$ making the following diagram commute
\begin{equation}\label{eq:diagram-big}
\xymatrix{
\mathrm{Sm}(k) \ar[rrr]^-{X\mapsto \perf_\dg(X)} \ar[d]_-{M(-)_\bbQ} &&& \dgcat(k) \ar[d]^-{U(-)_\bbQ} \\
\mathrm{DM}_{\mathrm{gm}}(k)_\bbQ \ar[d]_-\pi &&& \mathrm{NMot}(k)_\bbQ \ar[d]^-{\underline{\mathrm{Hom}}(-,U(k)_\bbQ)}\\
\mathrm{DM}_{\mathrm{gm}}(k)_\bbQ/_{\!-\otimes \bbQ(1)[2]} \ar[rrr]_-{\Phi} &&& \mathrm{NMot}(k)_\bbQ\,,
}
\end{equation}
where $\uHom(-,-)$ stands for the internal Hom of the closed symmetric monoidal structure and $\bbQ(1)[2]$ for the Tate object. Since the functor $\pi$, resp. $\Phi$, is additive, resp. fully-faithful and additive, we hence conclude from the combination of Proposition \ref{prop:computation} with the above commutative diagram \eqref{eq:diagram-big} that 
\begin{equation}\label{eq:computation}
\pi(M(Q)_\bbQ) \simeq \pi(M(\widetilde{C})_\bbQ \oplus M(C)_\bbQ^{\oplus (d-2)})\,.
\end{equation}
By definition of the orbit category, there exist then morphisms 
$$ \mathrm{f}=\{f_n\}_{n \in \bbZ} \in \Hom_{\mathrm{DM}_{\mathrm{gm}}(k)_\bbQ}(M(Q)_\bbQ, (M(\widetilde{C})_\bbQ \oplus M(C)_\bbQ^{\oplus (d-1)})(n)[2n])$$
$$ \mathrm{g}=\{g_n\}_{n \in \bbZ} \in \Hom_{\mathrm{DM}_{\mathrm{gm}}(k)_\bbQ}(M(\widetilde{C})_\bbQ \oplus M(C)_\bbQ^{\oplus (d-1)}, M(Q)_\bbQ(n)[2n])$$
verifying the equalities $\mathrm{g}\circ \mathrm{f}=\id=\mathrm{f}\circ \mathrm{g}$; in order to simplify the exposition, we write $-(n)[2n]$ instead of $- \otimes \bbQ(1)[2]^{\otimes n}$. Moreover, only finitely many of these morphisms are non-zero. Let us choose an integer $N\gg 0$ such that $f_n=g_n=0$ for every $|n|>N$. The sets $\{f_n\,|\, -N \leq n \leq N\}$ and $\{g_{-n}(n)\,|\,-N\leq n \leq N\}$ give then rise to the following morphisms between mixed motives:
$$ \alpha\colon M(Q)_\bbQ \too \oplus_{n=-N}^{N} (M(\widetilde{C})_\bbQ \oplus M(C)_\bbQ^{\oplus (d-1)})(n)[2n]$$
$$\beta\colon \oplus_{n=-N}^{N} (M(\widetilde{C})_\bbQ \oplus M(C)_\bbQ^{\oplus (d-1)})(n)[2n] \too M(Q)_\bbQ\,.$$
The composition $\beta\circ \alpha$ agrees with the $0^{\mathrm{th}}$ component of $\mathrm{g}\circ \mathrm{f}=\id$, \ie with the identity of $M(Q)_\bbQ$. Consequently, $M(Q)_\bbQ$ is a direct summand of the direct sum $\oplus_{n=-N}^{N} (M(\widetilde{C})_\bbQ \oplus M(C)^{\oplus (d-1)}_\bbQ)(n)[2n]$. Using the fact that $M(\widetilde{C})_\bbQ$ and $M(C)_\bbQ$ are both Kimura-finite, that $\wedge^2(\bbQ(1)[2])=0$, and that Kimura-finiteness is stable under direct sums, direct summands, and tensor products, we hence conclude that the mixed motive $M(Q)_\bbQ$ is also Kimura-finite. This finishes the proof of the first claim. Let us now prove the second claim.

Let $X$ be a smooth $k$-scheme whose mixed motive $M(X)_\bbQ$ is Kimura-finite. Note that since the functor $\pi$ is symmetric monoidal and additive, the object $\pi(M(X)_\bbQ)$ of the orbit category $\mathrm{DM}_{\mathrm{gm}}(k)_\bbQ/_{\!-\otimes \bbQ(1)[2]}$ is also Kimura-finite. As explained in \cite[\S3]{Bourbaki}, we have the following equality
$$ \mathrm{kim}(M(X)_\bbQ) = \chi(M(X)_{\bbQ, +}) - \chi(M(X)_{\bbQ, -})\,,$$
where $\chi$ stands for the Euler characteristic computed in the rigid symmetric monoidal category $\mathrm{DM}_{\mathrm{gm}}(k)_\bbQ$. Therefore, since the functor $\pi$ is moreover faithful, we observe that $\mathrm{kim}(M(X)_\bbQ)=\mathrm{kim}(\pi(M(X)_\bbQ))$. This leads to the following equalities:
\begin{eqnarray}\label{eq:equalities}
\mathrm{kim}(M(?)_\bbQ)=\mathrm{kim}(\pi(M(?)_\bbQ)) && ? \in \{Q, \widetilde{C}, C\}\,.
\end{eqnarray}
The Kimura-dimension of a direct sum of Kimura-finite objects is equal to the sum of the Kimura-dimension of each one of the objects. Hence, using the above computation \eqref{eq:computation} and the fact that the functor $\pi$ is additive, we conclude that 
\begin{equation}\label{eq:equalities1}
\mathrm{kim}(\pi(M(Q)_\bbQ))= \mathrm{kim}(\pi(M(\widetilde{C})_\bbQ))+ (d-1)\mathrm{kim}(\pi(M(C)_\bbQ))\,.
\end{equation}
The proof of the second claim follows now from the above equalities \eqref{eq:equalities}-\eqref{eq:equalities1}.

{\bf Item (ii).} The proof is similar to the one of item (i): simply replace $\widetilde{C}$ by $D$, $(d-1)$ by $(d-2)$, and use the fact that $\mathrm{kim}(M(D)_\bbQ)=\# D$.

\medbreak\noindent\textbf{Acknowledgments:} The author would like to thank Michel Van den Bergh for his interest in these results/arguments.

\end{document}

\end{proof}